\documentclass{gen-j-l2}
\usepackage{graphicx}
\usepackage{amsmath,amssymb}

\copyrightinfo{2009}{American Mathematical Society}

\newtheorem{theorem}{Theorem}[section]
\newtheorem{lemma}[theorem]{Lemma}
\newtheorem{conjecture}[theorem]{Conjecture}

\theoremstyle{definition}
\newtheorem{definition}[theorem]{Definition}

\theoremstyle{remark}

\numberwithin{equation}{section}

\begin{document}

\title{Knots and surfaces}

\author{Makoto Ozawa}
\address{Department of Natural Sciences, Faculty of Arts and Sciences, Komazawa University, 1-23-1 Komazawa, Setagaya-ku, Tokyo, 154-8525, Japan}
\email{w3c@komazawa-u.ac.jp}


\date{}

\dedicatory{}


\begin{abstract}
This article is an English translation of Japanese article "Musubime to Kyokumen", Math. Soc. Japan, Sugaku Vol. 67, No. 4 (2015) 403--423.
It surveys a specific area in Knot Theory concerning surfaces in knot exteriors.

In version 2, we added comments on the solutions or counterexamples for Conjecture 3.5, Conjecture 3.7 and Conjecture 5.30.
\end{abstract}

\maketitle





\section{Introduction}

A {\em knot} is either an embedding of the $1$-dimensional sphere into the $3$-dimensional sphere or its image\footnote{We shall work on the piecewise linear category, the differential category, or locally flat topological category.}.
Two knots are {\em equivalent} if they are mutually transformed by an orientation-preserving self homeomorphism of the $3$-dimensional sphere.
It is a fundamental problem on Knot Theory to determine whether given two knots are equivalent.
Since a knot itself is homeomorphic to the $1$-dimensional sphere, we cannot distinguish them by itself.
Then we observe the exterior of the knot.
If two knots are equivalent, then their exteriors are orientation-preserving homeomorphic. 
Conversely, by the Gordon--Luecke's knot complement theorem (\cite{GL1989}), if the exteriors of two knots are orientation-preserving homeomorphic, then they are equivalent.
Therefore, the fundamental problem on Knot Theory arrives at the homeomorphism problem of knot exteriors.

Since the homeomorphism problem of knot exteriors is a intrinsic problem, it is an efficient decisive condition what can exist in the knot exterior.
Then we shall consider an embedding of a surface into the knot exterior as we consider a simple closed curve on a surface.
If two knot exteriors are homeomorphic, then an embedding of a surface into one knot exterior which has some property moves to one of another knot exterior.
Therefore, we can consider the homeomorphism problem of knot exteriors by considering an embedding of a surface into the knot exterior.
Thus, it is an efficient means for the fundamental problem on Knot Theory to consider a problem of embeddings of surfaces into knot exteriors toward the embedding problem as knots.

The transversality theorem is most fundamental and important for handling surfaces embedded in knot exteriors.
By the transversality theorem, the intersection between two surfaces embedded in a knot exterior can be transverse, and it becomes a $1$-dimensional manifold.
This is a point of departure for research in surfaces embedded in knot exteriors.
If two surfaces are essential in the knot exterior, then the $1$-manifold obtained as the intersection of them can be essential in both surfaces (The essential transversality theorem).
In the case that two surfaces intersect essentially in this way, one can often reach an efficient conclusion since the $1$-manifold obtained as the intersection of them has a limitation on the surface.

There is no intersection if only one essential surface is given in the knot exterior.
In this case, we try to find some other convenient surface and consider the intersection with it.
For any knot, by using a Morse function on the $3$-sphere or the knot exterior, one can consider the intersection with the Heegaard surface as a level surface or the bridge decomposing sphere.
In this case, the intersection with the level surface can be essential and all saddle points can be essential (essential Morse position).
Also, for alternating knots, there is a standard position which is introduced by Menasco (\cite{M1984}).
By inserting small $3$-dimensional balls (bubbles) in each crossing on a regular diagram of a knot, one can consider the intersection with the $2$-sphere on which the regular diagram lies and the boundary of bubbles.

It is still efficient to consider two surfaces and their intersection even if no essential surface is given in the knot exterior.
For example, the Gordon--Luecke's knot complement theorem was proved by using the level sphere obtained from some minimal one among Morse positions (thin position) as the settlement of the property R conjecture by Gabai (\cite{G1987a}).
Specifically, by supposing that the theorem didn't hold and suppose that there exist two knots which have the homeomorphic exteriors, we have level spheres which are obtained from each Morse positions.
By a deep consideration on the graph which is obtained from the intersection between two planar surfaces which are obtained by restricting them on the knot exteriors, a contradiction can be reached.
Also, the Tait's flipping conjecture that any reduced alternating regular diagrams of a prime alternating knot are mutually transformed by flippings was proved by Menasco--Thistlethwaite by considering the intersection between two checkerboard surfaces which are obtained from two alternating regular diagrams (\cite{MT1993}).

We stated as above the case that two essential surfaces or only one essential surface are given and the case that no essential surface is given.
Next, let's consider that case that many surfaces are contained without intersection.
One can consider that the knot exterior is a large space and the knot itself is complicated as much as there exist mutually disjoint, mutually not parallel many essential surfaces in the knot exterior.
Can we take the number of such essential surfaces as infinitely large?
In fact, by the Kneser--Haken's finiteness theorem (\cite{K1929}, \cite{H1961a}), if the number of mutually disjoint essential surfaces in a $3$-manifold exceeds a constant number depending on the manifold, then they become to be mutually parallel.
This constant is called a Haken number, and it is a standard for measuring the range of manifolds.
In general, there are infinitely many essential surfaces in a $3$-manifold.
But Floyd--Oertel showed that there exists a set of finitely many branched surfaces such that any essential surface can be obtained from their branched surfaces (\cite{FO1984}).
Therefore, a study of essential surfaces in knot exteriors arrives at a study of branched surfaces which give them.
By Hatcher--Thurston, for essential surfaces in the $2$-bridge knot exteriors, branched surfaces are given (\cite{HT1985}).

The criterion of the fundamental classification on essential surfaces in knot exteriors is either the number of boundary components, the boundary slopes, and the genus etc.
The surfaces mainly dealt on study of knots are classified into the following four types on the relation between a knot $K$ and a surface $F$.
They are a closed surface which is disjoint from the knot ($K\cap F=\emptyset$), a tangle decomposing sphere which intersects the knot transversely ($K \pitchfork F$), a Seifert surface whose boundary is the knot ($K=\partial F$), and a closed surface which contains the knot, called a coiled surface ($K\subset F$).
The above four types of surfaces are summarized in the below table.
For each types, we can consider essential surfaces and Heegaard surfaces, and described surfaces and invariants which are related.
Moreover, for the first three Heegaard surfaces, the related Morse functions are cited.

\medskip

\begin{tabular}{l|l|ll}
& essential surfaces & Heegaard surfaces & Morse functions\\
\hline
$K\cap F=\emptyset$ & closed surface & tunnel number & $E(K)\to \mathbb{R}_+$ \\
$K\pitchfork F$ & tangle sphere & bridge number & $(S^3,K)\to \mathbb{R}$\\
$K=\partial F$ & Seifert surface & Morse--Novikov number & $E(K)\to S^1$ \\
$K\subset F$ & coiled surface & $h$-genus & \\
\end{tabular}

\medskip

We assume an elementary knowledge of topological spaces, manifolds, fundamental groups, homology groups, Morse theory, handle decompositions, Heegaard splittings for reading this manuscript (e.g. \cite{M2002}, \cite{S1993}).
And a knot in the $3$-sphere is denoted by $K$, $N(*)$,  $|*|$, $int(*)$, $E(K)=S^3-int N(K)$ denote a regular neighborhood, the number of connected components, the interior, the exterior of $K$ respectively.

\section{Fundamentals of essential surfaces}

\subsection{Essential surfaces}

\begin{definition}[Essential $1$-dimensional manifolds]

We say that a loop $\alpha$ properly embedded in a $2$-dimensional manifold $F$ is {\em inessential} if there exists a disk $D$ in $F$ such that $\partial D=\alpha$.
If $\alpha$ is not inessential, then we say that it is {\em essential}.
And we say that an arc $\alpha$ properly embedded in $F$ is {\em inessential} if there exists a disk $D$ in $F$ such that $\partial D=\alpha\cup \beta,\ \alpha\cap\beta=\partial \alpha=\partial \beta$ ($\beta$ is an arc in $\partial F$).
If $\alpha$ is not inessential, then we say that it is {\em essential}.
\end{definition}

Hereafter, we assume that a $3$-dimensional manifold $M$ is orientable.

\begin{definition}[Incompressible surfaces]
We say that an orientable surface $F$ properly embedded in a $3$-dimensional manifold $M$ is {\em compressible} if when $F$ is a disk, there exist a disk $D$ in $\partial M$ and a $3$-dimensional ball $B$ in $M$ such that $\partial F=\partial D$ and $\partial B=F\cup D$, when $F$ is a sphere, there exists a $3$-dimensional ball $B$ in $M$ such that $\partial B=F$, and otherwise, there exists a disk $D$ in $M$ such that $D\cap F=\partial D$ and $\partial D$ is an essential loop in $F$.
If $F$ is not compressible, then we say that it is {\em incompressible}\footnote{A two-sided surface is incompressible if and only if the homomorphism of the fundamental group induced by an inclusion map of the surface is injective.}.
\end{definition}

\begin{definition}[Boundary incompressible surfaces]
We say that an orientable surface properly embedded in a $3$-dimensional manifold $M$ is {\em boundary compressible} if there exists a disk $D$ in $M$ such that $D\cap F=\partial D\cap F=\alpha$ is an essential arc properly embedded in $F$ and $D\cap \partial M=\partial D-\rm{int}\alpha$ is an arc in $\partial M$.
If $F$ is not boundary compressible, then we say that it is {\em boundary incompressible}.
\end{definition}

\begin{definition}[Boundary parallel]
We say that an orientable surface properly embedded in a $3$-dimensional manifold $M$ is {\em boundary parallel} if there exists an embedding $h:F\times [0,1]\to M$ such that $h(F\times \{0\})=F$ and $h(F\times [0,1])\cap \partial M=h(\partial F\times [0,1]\cup F\times \{1\})$.
Namely, $F$ is isotopic to a subsurface in $\partial M$.
\end{definition}

\begin{definition}[Essential surfaces]
We say that an orientable surface properly embedded in a $3$-dimensional manifold $M$ is {\em essential} if $F$ is incompressible and boundary incompressible, and not boundary parallel.
For a non-orientable surface $F\subset M$, we say that $F$ is {\em essential} if an orientable surface $\partial N(F)$ is incompressible and boundary incompressible, and not boundary parallel in $M$.
\end{definition}

\begin{definition}[Irreducible and boundary irreducible]
We say that a $3$-dimensional manifold $M$ is {\em irreducible} if there does not exist an incompressible sphere.
And we say that $M$ is {\em boundary irreducible} if there does not exist an incompressible disk in $M$.
\end{definition}

\begin{definition}[Peripherally incompressible]
We say that an essential surface $F$ in a knot exterior $E(K)$ is {\em peripherally incompressible} if there exists an annulus $A$ which connects an essential and not boundary parallel loop in $F$ and a loop in $\partial E(K)$, and then $A$ is called a {\em peripherally compressing annulus}.
\end{definition}

\subsection{Essential transversality theorem}

\begin{definition}[General position]
We say that $n_i$-dimensional manifolds $X_i$ $(i=1,2)$ properly embedded in an $m$-dimensional manifold $Y$ {\em intersect transversely} if for any point $p\in X_1\cap X_2$, there exists a neighborhood $U$ such that $(U,U\cap X_1,U\cap X_2)$ is homeomorphic to $(\mathbb{R}^m,\mathbb{R}^{n_1}\times \{0\}^{m-n_1},\{0\}^{m-n_2}\times \mathbb{R}^{n_2})$ or $(\mathbb{R}^{m-1}\times\mathbb{R}_+,\mathbb{R}^{n_1}\times \{0\}^{m-n_1-1}\times\mathbb{R}_+,\{0\}^{m-n_2-1}\times \mathbb{R}^{n_2}\times\mathbb{R}_+)$.
We say that $X_1$ and $X_2$ are in {\em general position} if $X_1\cap X_2=\emptyset$ or $X_1$ and $X_2$ intersect transversely in an $(n_1+n_2-m)$-dimensional submanifold.
\end{definition}

\begin{theorem}[Transversality theorem cf. \cite{GP2010}]\label{general position}
For any $n_i$-dimensional manifolds $X_i$ $(i=1,2)$ properly embedded in an $m$-dimensional manifold, there exists an isotopy of $X_1$ so that $X_1$ and $X_2$ are in general position.
\end{theorem}

\begin{theorem}[Essential transversality theorem]\label{intersection}
Let $M$ be an irreducible and boundary irreducible $3$-dimensional manifold, and $F_1$ and $F_2$ be an incompressible and boundary incompressible surfaces properly embedded in $M$.
Then, there exist isotopies of $F_1$ and $F_2$ so that each component of $F_1\cap F_2$ is a loop or arc which is essential in both $F_1$ and $F_2$\footnote{An intersection between strongly irreducible Heegaard surfaces can be essential (\cite{RS1996}, \cite{KS2000})}.
\end{theorem}

\begin{proof}
By Theorem \ref{general position}, we may assume that $F_1$ and $F_2$ are in general position.
Then, $F_1\cap F_2$ is a $1$-dimensional manifold, that is, it consists of loops and arcs.
Take $|F_1\cap F_2|$ is minimal under isotopies of $F_1$ and $F_2$.

First, suppose that there exists a loop component of $F_1\cap F_2$ which is inessential in $F_1$.
Let $\alpha$ be an {\em innermost loop} in $F_1$, $\delta$ be a disk in $F_1$ satisfying $\partial \delta =\alpha$.
Since this disk $\delta$ satisfies that $\delta\cap F_2=\partial \delta$, by the incompressibility of $F_2$, $\partial \delta$ is an inessential loop in $F_2$.
Hence, there exists a disk $\delta'$ in $F_2$ such that $\partial \delta'=\alpha$.
Here, we remark that since $\delta$ is an innermost disk, $\text{int}\delta$ does not intersect $F_2$, but $\delta'$ may intersect $F_1$.

A $2$-sphere $S=\delta\cup \delta'$ is obtained from these two disks $\delta$ and $\delta'$.
By the irreducibility of $M$, there exists a $3$-ball $B$ such that $\partial B=S$.
If $\text{int}B\cap F_2=\emptyset$, then we can remove a loop component $\alpha$ of $F_1\cap F_2$ by an isotopy from $\delta'$ to $\delta$ along $B$.
This contradicts the minimality of $|F_1\cap F_2|$.
Otherwise, if $\text{int}B\cap F_2\ne\emptyset$, namely, in the case that $B\supset F_2$, it contradicts Theorem \ref{handlebody}.
Therefore, every loop component of $F_1\cap F_2$ is essential in $F_1$.
Similarly, one can prove that it is essential in $F_2$.

Next, suppose that there exists an arc component of $F_1\cap F_2$ which is inessential in $F_1$.
Let $\alpha$ be an {\em outermost arc} in $F_1$, $\delta$ be a disk in $F_1$ satisfying $\partial \delta =\alpha\cup \beta$ ($\beta$ is an arc in $\partial F_1$).
Since this disk $\delta$ satisfies that $\delta\cap F_2=\alpha$ and $\delta\cap \partial M=\beta$, by the boundary incompressibility of $F_2$, $\alpha$ is an inessential arc in $F_2$.
Hence, there exists a disk $\delta'$ in $F_2$ such that $\partial \delta'=\alpha\cup \beta'$ ($\beta'$ is an arc in $\partial F_2$).
Here, we remark that since $\delta$ is an outermost disk, $\text{int}\delta$ does not intersect $F_2$, but  $\delta'$ may intersect $F_1$.

A disk $D=\delta\cup \delta'$ is obtained from these two disks $\delta$ and $\delta'$.
By the boundary irreducibility of $M$, there exists a $3$-ball $B$ in $M$ such that $\partial B=D\cup D'$ ($D'$ is a disk in $\partial M$).
If $\text{int}B\cap F_2=\emptyset$, then we can remove an arc component $\alpha$ of $F_1\cap F_2$ by an isotopy from $\delta'$ to $\delta$ along $B$.
This contradicts the minimality of $|F_1\cap F_2|$.
Otherwise, if $\text{int}B\cap F_2\ne\emptyset$, namely, in the case that $B\supset F_2$, it contradicts Theorem \ref{handlebody}.
Therefore, every arc component of $F_1\cap F_2$ is essential in $F_1$.
Similarly, one can prove that it is essential in $F_2$.
\end{proof}

\begin{theorem}[{\cite[Proposition 4.8]{MS1994}}, \cite{W1968}, cf. \cite{E1961}]
Let $M$ be an irreducible and boundary irreducible $3$-manifold, $F_1$ and $F_2$ be incompressible and boundary incompressible surfaces properly embedded in $M$.
If $F_1$ and $F_2$ are in general position, then
\begin{enumerate}
\item if $F_1$ is isotopic to $F_2$, then there exists a parallel region between $F_1$ and $F_2$.
\item if $F_1\cap F_2\ne\emptyset$, and $F_1$ is isotopic to a surface disjoint from $F_2$, there exists a parallel region between $F_1$ and $F_2$.
\end{enumerate}
\end{theorem}

By the argument similar to the proof of Theorem \ref{intersection} with respect to a level sphere for the standard Morse function of the $3$-sphere, one can eliminate an innermost inessential saddle, and hence can show the next theorem.

\begin{theorem}[\cite{F1948}, \cite{H1954} cf. \cite{A1924}, {\cite[Theorem 1.1]{AH1}}]\label{3-sphere}
The $3$-sphere does not contain an essential surface.
\end{theorem}

If the $3$-ball is obtained from a $3$-manifold by cutting it along essential disks, then it is called a {\em handlebody}.
By considering the intersection with essential disks of a handlebody, one can show the next theorem as well as Theorem \ref{intersection}.

\begin{theorem}[\cite{BO1983}, \cite{HS2001}]\label{handlebody}
An essential surface in a handlebody is only a disk\footnote{Moreover, it was shown that an essential surface in a compression body, which is a generalization of a handlebody (\cite{FB1983}), is only a disk or a vertical annulus ({\cite[Lemme 9]{BO1983}}, {\cite[Lemma 2.4]{HS2001}}).}.
\end{theorem}

\section{Position of knots}

\subsection{Bridge positions and Morse positions ($(S^3,K)\to \mathbb{R}$)}

A knot $K$ is in {\em Morse position} with respect to the standard Morse function $h:S^3\to \mathbb{R}$ if $h|_K$ is a Morse function\footnote{In general, for a knot in a $3$-manifold, Morse position can be defined by any Morse function.}.
Hereafter, we assume that a knot $K$ is in Morse position.
Let $t_1< \cdots< t_n$ be critical value of $h|_K$, and choose $r_i\in \mathbb{R}$ $(i=1,\ldots,n-1)$ so that $t_i<r_i<t_{i+1}$.
We say that a level sphere $S_{r_i}=h^{-1}(r_i)$ $(1\le i\le n-1)$ is a {\em thick sphere} if $|S_{r_{i-1}}\cap K|<|S_{r_{i}}\cap K|$ and $|S_{r_{i}}\cap K|>|S_{r_{i+1}}\cap K|$.
We say that a level sphere $S_{r_i}=h^{-1}(r_i)$ is a {\em thin sphere} if $|S_{r_{i-1}}\cap K|>|S_{r_{i}}\cap K|$ and $|S_{r_{i}}\cap K|<|S_{r_{i+1}}\cap K|$.
We say that a knot $K$ is in {\em bridge position} if it is in Morse position without thin level spheres.
Then, $K$ has only one thick level sphere, and it separates all maximal points from all minimal points of $K$.
This level sphere is called a {\em bridge decomposing sphere}.
The minimum number of maximal points among all bridge position is called the {\em bridge number}, and denoted by $b(K)$ (\cite{Sch1954}).

A disk $D$ satisfying the following is called an {\em upper (lower) disk} of $K$ with respect to a level sphere $S_{r_i}$.
The boundary of $D$ is divided into two arcs as $\partial D=\alpha\cup\beta$, $\alpha$ is a subarc of $K$ which contains exactly one maximal (minimal) point, and $\beta$ is an arc on $S_{r_i}$.
An upper (lower) disk is called a {\em strongly upper (lower) disk} if it does not intersect all thick and thin spheres in its interior.
If there exist a strongly upper disk and a strongly lower disk which have a common one point of $K$ with respect to a level sphere, then there exists an isotopy which cancels their maximal point and minimal point.
We call this a {\em Type I move}.
And if there exist a strongly upper disk and a strongly lower disk which are mutually disjoint, then there exists an isotopy which interchanges their maximal point and minimal point.
We call this a {\em Type II move}.
There are theorems for bridge positions and Morse positions which correspond to the Reidemeister's theorem for regular diagrams of knots (\cite{R1926}, \cite{AB1926}).

\begin{theorem}[Stably equivalent theorem for bridge positions \cite{B1976}, \cite{H1998}]\label{fundamental1}
Two knots are equivalent if and only if their bridge positions are mutually moved by a sequence of Type I moves and its reverse moves.
\end{theorem}

\begin{theorem}[Stably equivalent theorem for Morse positions \cite{S2009}]\label{fundamental2}
Two knots are equivalent if and only if their Morse positions are mutually moved by a sequence of Type I and II moves and their reverse moves.
\end{theorem}

\begin{definition}[\cite{G1987a}]
We define the {\em width} $w(K)$ of a knot $K$ as
\[
w(K)=\min_{K} \sum_{i=1}^{n-1}|S_{r_i}\cap K|,
\]
where $\min$ is taken over all Morse positions of $K$.
We say that a knot $K$ is in a {\em thin position} if it realizes $w(K)$.
And we call a Morse position which does not admit Type I and II moves a {\em locally thin position} (cf. \cite{HK2003}).
\end{definition}

The bridge number satisfies the additivity with respect to a connected sum of knots (\cite{Sch1954}, \cite{Sch2003}, cf. \cite{D1992}, \cite{S2014}), but the width does not satisfy the additivity (\cite{BT2013}).
In general, it holds that $w(K_1\# K_2)\ge \max \{w(K_1),w(K_2) \}$ (\cite{SS2006}).

\begin{definition}[\cite{O1998k}]
We define the height $ht(K)$ of a knot $K$ as
\[
ht(K)=\max_{K} \{|\{\text{thick sphere}\}|\},
\]
where $\max$ is taken over all thin positions of $K$.
\end{definition}

\begin{conjecture}\label{additive_height}
It holds that $ht(K_1\# K_2) = ht(K_1)+ht(K_2)$\footnote{In \cite{BO2017}, a counterexample was given.}.
\end{conjecture}

\begin{definition}[\cite{MOz2010}]
We define the {\em trunk} of a knot $K$ as
\[
trunk(K)=\min_{K} \max_{t\in \mathbb{R}} |h^{-1}(t)\cap K|,
\]
where $\min$ is taken over all Morse positions of $K$.
\end{definition}

\begin{conjecture}[\cite{MOz2010}]\label{additive_trunk}
It holds that $trunk(K_1\# K_2)=\max\{trunk(K_1),$ $trunk(K_2)\}$\footnote{This conjecture was solved affirmatively in \cite{DZ2016}.}.
\end{conjecture}

Conjectures \ref{additive_height} and \ref{additive_trunk} hold for meridionally small knots, namely, knots without meridians as their boundary slopes ({\cite[Theorem 1.8]{MOz2010}}).

Let $S$ be a bridge decomposing sphere for a knot $K$ in bridge position, and $S$ divide $S^3$ into two $3$-balls $B_+, B_-$.
Denote the curve complex obtained from isotopy classes of essential and not boundary parallel loops in $S\cap E(K)$ by $\mathcal{C}(S)$.
And denote the subcomplex of $\mathcal{C}(S)$ obtained from isotopy classes of $\partial D$, where $D$ is an essential disk properly embedded in $B_{\pm}-K$ by $\mathcal{C}(B_{\pm})$.

\begin{definition}[\cite{BS2005}, \cite{T2007}, cf. \cite{J2014}, \cite{H2001}]
We define the {\em Hempel distance} of $K$ with respect to $S$ as
\[
d(K,S)=\min \{ d_{\mathcal{C}(S)}(x,y) | x\in \mathcal{C}(B_+), y\in \mathcal{C}(B_-)\},
\]
where $d_{\mathcal{C}(S)}$ denotes the distance in $\mathcal{C}(S)$.
\end{definition}

\subsection{Tunnel numbers ($E(K)\to \mathbb{R}_+$)}

When we have a handle decomposition (\cite{M2002}) of the exterior $E(K)$ of a knot $K$ as 
$$E(K)=N(\partial E(K); E(K)) \cup (1\text{-handles}) \cup (2\text{-handles}) \cup (3\text{-handle}),$$
we call the core of a $1$-handle an {\em unknotting tunnel}.
We define the {\em tunnel number} of $K$, denoted by $t(K)$, as the minimal number of unknotting tunnels (\cite{C1980}).
In general, the tunnel number satisfies that $t(K_1\# K_2)\le t(K_1)+t(K_2)+1$ with respect to the connected sum of knots (\cite{C1980}, cf. \cite{S2014}).
In this inequality, there exists an example satisfying the equality (\cite{MR1997}, \cite{MSY1996}), on the other hand, there exists an example having an arbitrarily large gap (\cite{K1994}, \cite{M1995}).
And it holds that $t(K_1\# K_2)\ge t(K_1)+t(K_2)$ for small knots and meridionally small knots which include them (\cite{MS2000}, \cite{M2000}).
It has been recently proved that $t(K_1\# K_2)\ge \max \{ t(K_1),t(K_2)\}$ in general (\cite{TS2014}).

Concerning the rank of the fundamental group and the Heegaard genus of closed $3$-manifolds $M$, it holds that $rank(\pi_1(M))\le g(M)$.
Waldhausen proposed a problem whether does the equality hold in this inequality.
This problem is true if $rank(\pi_1(M))=0$ since it is Poincar\'{e} conjecture, but in general, some examples which do not satisfy the equality (\cite{BZ1984}, \cite{SW2007}).
However, Waldhausen's problem is unsolved for the knot exteriors.

\begin{conjecture}[The rank versus genus conjecture, cf. {\cite[Question 2]{L2013}}]\label{rank_genus}
It holds that $rank(\pi_1(E(K)))=t(K)+1$.
\end{conjecture}

For a $2$-generator knot, namely, a knot satisfying $rank(\pi_1(E(K)))=2$, if it is a cabled knot, then $t(K)=1$ (\cite{B1994}).
It is expected that a $2$-generator satellite knot is either tunnel number one or some satellite knot, and the latter does not exist ({\cite[Corollary 7]{BW2005}}, cf. \cite{BJ2004}).

In connection with Conjecture \ref{rank_genus}, there is a next conjecture.

\begin{conjecture}[The meridional rank versus bridge number conjecture, {\cite[Problem 1.11]{K1995}}, cf. \cite{BJ2004}]\label{meridian_bridge}
The minimum number of meridional generators of $\pi_1(E(K))$ is equal to the bridge number of $K$.
\end{conjecture}

Conjecture \ref{meridian_bridge} holds for $2$-bridge knots (\cite{BZ1989}), Montesinos knots (\cite{BZ1985}), torus knots (\cite{RZ1984}), cable knots of torus knots (\cite{BJW2014}).
And if for a knot whose $2$-fold branched cover is a graph manifold, $\pi_1(E(K))$ is generated by the $3$ meridians, then it is a $3$-bridge knot (\cite{BJW2014}).
A new approach to Conjecture \ref{meridian_bridge} was given in \cite{BKVV2017}.

\subsection{Morse--Novikov number ($E(K)\to S^1$)}\label{Morse--Novikov}

When we have a handle decomposition of the exterior $M = E(K) - int (F \times [0,1])$ of a Seifert surface $F$ as 
$$M = N(F \times \{0\}) \cup (1\text{-handles}) \cup (2\text{-handles}) \cup N(F \times \{1\}),$$
we define $h(F)$ as the minimum number of $1$-handles.
For a knot $K$, we call the minimum number of $h(F)$ among all Seifert surfaces $F$ the {\em handle number} of $K$, and denoted by $h(K)$ (\cite{G1992}).
Knots with $h(K)=0$ are called fibered knots.
On the other hand, for a Morse function $E(K)\to S^1$ satisfying some conditions, the minimum number of critical points is called a {\em Morse--Novikov number}, and denoted by $MN(K)$ (\cite{PRW2002}).
For the handle number and the Morse--Novikov number, it holds that $MN(K)=2h(K)$ (\cite{GP2005}).

\begin{conjecture}[\cite{PRW2002}]
It holds that $MN(K_1\# K_2) = MN(K_1)+MN(K_2)$.
\end{conjecture}

The handle number and the Morse--Novikov number are additive with respect to the connected sum of almost small knots (\cite{MG2013}).
It also holds that $MN(K)\le 2t(K)$ (\cite{P2010}).

\subsection{Regular diagrams ($(S^3,K)\to S^2$)}

Concerning principal classes of knots obtained from regular diagrams, there is an inclusion relation as follows (cf. \cite{O2011}, \cite{MO2010}).

\begin{figure}[htbp]
\begin{center}
\includegraphics[trim=0mm 0mm 0mm 0mm, width=.6\linewidth]{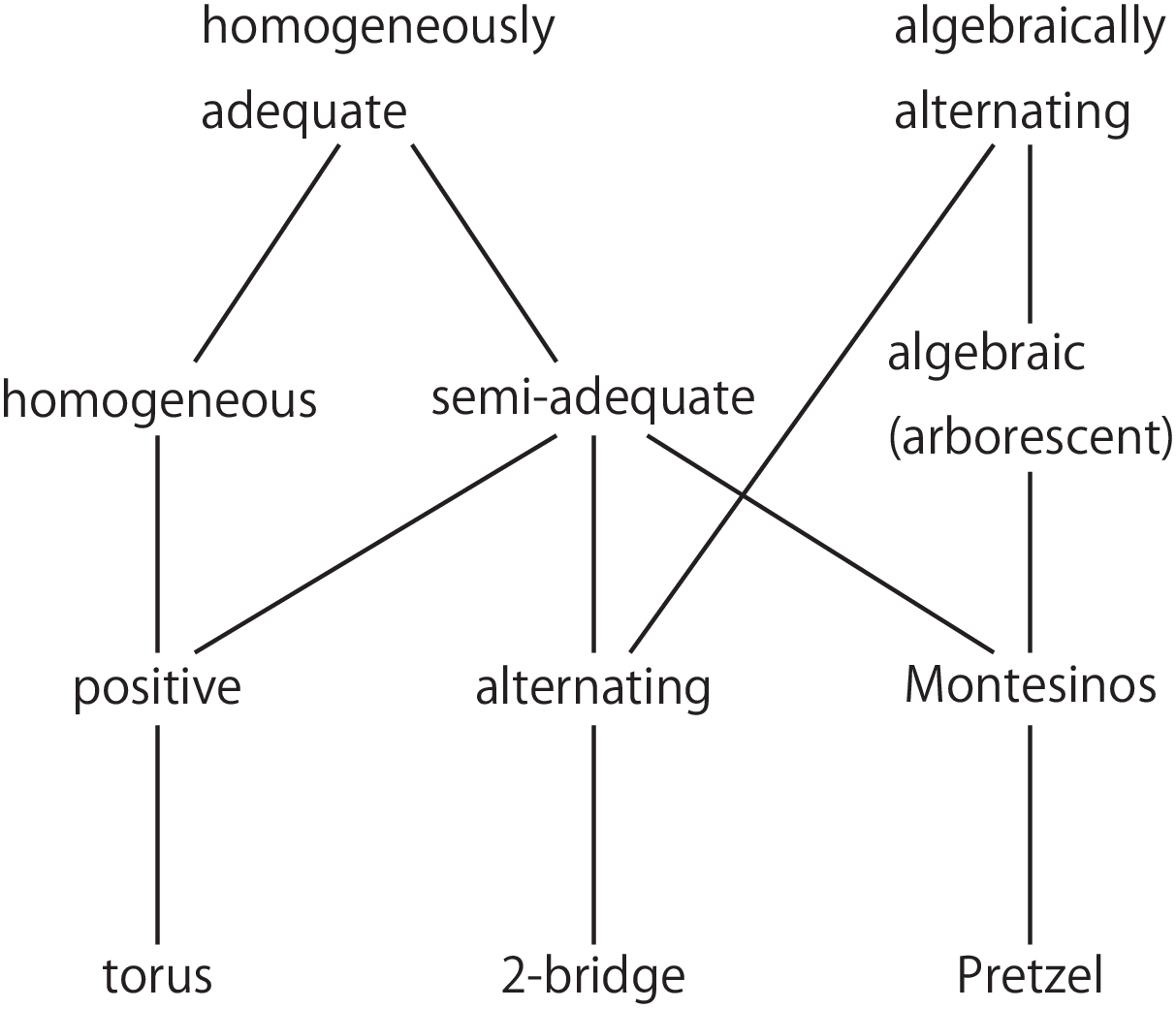}
\caption{A Hasse diagram of regular diagrams of knots}
\label{}
\end{center}
\end{figure}

\section{Fundamental theorems on knots}

The next knot complement theorem shows that embeddings of a knot exterior into the $3$-sphere are unique.

\begin{theorem}[Knot complement theorem \cite{GL1989}]
Two knots are equivalent if and only if their complements are orientation-preserving homeomorphic\footnote{Two prime knot complements are homeomorphic if and only if their fundamental groups are homomorphic.}.
\end{theorem}

A knot is the {\em trivial knot} only if the knot exterior contains an essential disk.
A knot is said to be {\em satellite} if its exterior contains an essential torus.
A {\em torus knot} is characterized as a knot whose exterior contains an essential annulus and does not contain an essential torus (cf. Theorem \ref{classify_annuli}).

\begin{theorem}[Knot hyperbolization theorem \cite{T1982}]
Any knot is either the trivial knot, a torus knot, a satellite knot, or a hyperbolic knot.
\end{theorem}

By Theorem \ref{distance}, if $d(K,F)\ge 3$ with respect to a bridge decomposing sphere $F$, then $K$ is a hyperbolic knot.

We say that for a sphere $S$ which intersects a knot $K$ in $2$ points transversely, $S$ is a {\em decomposing sphere} for $K$ if the annulus $S\cap E(K)$ is essential.
Then, $K$ is decomposed into two knots $K_1$ and $K_2$.
Conversely, we say that $K$ is obtained by a {\em connected sum} of $K_1$ and $K_2$, and denoted by $K=K_1\# K_2$.
We say that a knot $K$ is {\em prime} if $K$ is not the trivial knot and does not have a decomposing sphere.
A theorem similar to the uniqueness of prime factorization for the integers also holds for the knots.

\begin{theorem}[Uniqueness of prime decompositions \cite{S1949}, cf. \cite{H1958}]
Every non-trivial knot is uniquely decomposed into a connected sum of prime knots.
\end{theorem}

In general, there can exist infinitely many essential surfaces of arbitrarily large genus in a knot exterior (\cite{MSS2006}, \cite{W2008}, \cite{K1991}, \cite{HL1971}, \cite{RG1981}, \cite{RG1994}, \cite{OT2003}), but if we restrict the genus, then the finiteness follows.

\begin{theorem}[\cite{Oer2002}, cf. \cite{JO1984}]\label{finiteness}
In a hyperbolic knot exterior, for any natural number $N$, the number of essential surfaces of genus less than or equal to $N$ is finite.
\end{theorem}

Every non-trivial knot exterior contains at least one both of non-separating essential surfaces and separating essential surfaces.

\begin{theorem}[\cite{FP1930}, \cite{S1934}]
For any knot exterior, there exists a non-separating essential surface with non-empty boundary.
\end{theorem}

\begin{theorem}[\cite{CS1984}]\label{separating}
For any non-trivial knot exterior, there exists a separating essential surface with non-empty boundary.
\end{theorem}

For any odd number $n$, there exists a knot whose exterior contains a non-separating essential surface with $n$ boundary components (\cite{E2013}).

\subsection{Boundary slopes}

\begin{definition}[Boundary slopes]
For a Seifert surface $F$ of a knot $K$, the isotopy class of $\partial(F\cap E(K))$ on $\partial E(K)$ is called a {\em longitude}\footnote{The uniqueness of a longitude and a meridian can be proved by using the parity rule ({\cite[Lemma 4.1]{HM1995}}, \cite{CGLS1987}).}.
And the isotopy class of an essential loop on $\partial E(K)$ which bounds a disk in $N(K)$ is called a {\em meridian}.
Th $1$-dimensional homology class of an essential loop $\alpha$ on $\partial E(K)$ can be represented as $[\alpha]=p[m]+q[l]$ by using a meridian $m$ and a longitude $l$.
The isotopy classes of essential loops on $\partial E(K)$ correspond to $\mathbb{Q}\cup \{1/0\}$ one-to-one by regarding as a rational number $p/q$ when $q\ne 0$, $1/0$ when $q=0$.
For an essential surface $F$ with boundary embedded in a knot exterior $E(K)$, a component $\alpha$ of $\partial F$ determines $p/q\in \mathbb{Q}\cup \{1/0\}$.
We call this $p/q$ the {\em boundary slope} of $F$.
The set of all boundary slopes of $K$ is denoted by $\mathcal{B}(K)$.
\end{definition}

\begin{theorem}[\cite{Hatcher1982}]
For any knot $K$, $\mathcal{B}(K)$ is a finite set.
\end{theorem}

By Theorem \ref{handlebody}, for the trivial knot, $\mathcal{B}(K)=\{0\}$.

\begin{theorem}[\cite{CS1984}]
For any non-trivial knot $K$, $\mathcal{B}(K)$ contains at least two elements.
\end{theorem}

Since the essential surfaces in the torus knot exterior are only the minimal genus Seifert surface and the cabling annulus (\cite{T1994}), we have $|\mathcal{B}(K)|=2$ for a torus knot $K$.
In general, since the more the number of essential surfaces increases as the knot becomes complicated, the more $|\mathcal{B}(K)|$ increases (\cite{HT1985}, \cite{HO1989}, \cite{ND2001}).

\begin{conjecture}[Two slopes conjecture \cite{Z1991}]
If $|\mathcal{B}(K)|=2$, then $K$ is a torus knot.
\end{conjecture}

\begin{theorem}[\cite{HO1989}]
For any rational number $p/q$, there exists a knot $K$ such that $p/q\in \mathcal{B}(K)$.
\end{theorem}

Some general evaluation formulae on boundary slopes for two surfaces in a knot exterior are given in \cite{HM1997}, \cite{T1996} etc.
And an evaluation on peripheral slopes of peripherally compressible essential closed surface is in \cite{IO2002}.
On the other hand, for any integer $m\ge0$, there exists a hyperbolic knot whose exterior contains $m+1$ peripherally compressible essential closed surfaces with peripheral slopes $0, 1, \ldots ,m$ (\cite{Mi2004}).
And for a set $\{a_1,\ldots,a_n\}$ of any even integers, there exists a hyperbolic knot which has an orientable essential spanning surface (namely, a surface without closed component, and its boundary is the knot) $F_i$ $(i=1,\ldots,n)$ with boundary slope $a_i$ (\cite{T2004}).

\subsection{Non-meridional essential planar surfaces and the cabling conjecture}

\begin{lemma}[{\cite[Lemma 1.10]{W1967}}, {\cite[Lemma 1.10]{AH1}}, {\cite[Proposition 2.1]{HM1995}}]
An incompressible and boundary compressible surface in a knot exterior is a boundary parallel annulus.
\end{lemma}

\begin{theorem}[The classification of essential annuli \cite{S1973}]\label{classify_annuli}
Let $K$ be a knot in $S^3$.
If $E(K)$ contains an essential annulus $A$, then one of the followings holds.
\begin{enumerate}
\item $K$ is a torus knot or a cabled knot, and $A$ is its cabling annulus.
\item $K$ is a composite knot, and $A$ can be extend into a decomposing sphere for $K$.
\end{enumerate}
\end{theorem}

\begin{conjecture}[The cabling conjecture \cite{GS1986}]\label{cabling}
If a reducible manifold is obtained by a Dehn surgery along a knot $K$ in $S^3$, then $K$ is a cabled knot, and the surgery slope is the boundary slope of the cabling annulus.
\end{conjecture}

The cabling conjecture was solved on satellite knots(\cite{S1990}), strongly invertible knots (\cite{EM1992}), alternating knots (\cite{MT1992}), symmetric knots (\cite{HS1998c}), knots with bridge decomposing sphere of Hempel distance greater than or equal to $3$ (\cite{BCJTT201211}).
In particular, since knots with Hempel distance of $2$ are classified (\cite{BCJTT201309}), it would be one way for the solution to consider the cabling conjecture on this class.
And if the manifold obtained by a Dehn surgery is a connected sum of lens spaces, then the cabling conjecture is true (\cite{G2010}).

\begin{conjecture}[Strong cabling conjecture]
A non-meridional essential planar surface in a knot exterior is either a disk or an annulus.
\end{conjecture}

A non-meridional essential planar surface $P$ properly embedded in a knot exterior $E(K)$ has an integral boundary slope (\cite{GL1987}). 
$|\partial P|$ is not an odd integer greater than or equal to $3$ ({\cite[Lemma 3.5]{KO2013}}), and not $4$ ({\cite[Lemma 3.6]{KO2013}}).

\section{Knots and surfaces}

\subsection{Closed surfaces ($K\cap F=\emptyset$)}

\subsubsection{Essential closed surfaces}

We say that a knot is {\em small} if there does not exist an essential closed surface in its exterior, and that it is {\em large} if there exists an essential closed surface.
Essential closed surfaces in the exterior of a large knot are divided into peripherally compressible and peripherally incompressible.
For a peripherally compressible essential closed surface, the peripheral slope is defined from a peripherally compressing annulus (\cite{IO2000}), it is either meridional or integral ({\cite[Lemma 2.5.3]{CGLS1987}}).
A tangle decomposing sphere is obtained from a meridional peripherally compressible essential closed surface, and a coiled surface is obtained from an integral peripherally compressible essential closed surface.
Conversely, if there exists a peripherally compressible essential closed surface with boundary slope $\gamma$, then there exists a peripherally compressible essential closed surface with peripheral slope $\gamma$ (\cite{OT2003}).

\begin{theorem}[\cite{T1994},\cite{HT1985},\cite{O1984},\cite{M2011}]
Torus knots, $2$-bridge knots, Montesinos knots with length $3$, $2$-twisted torus knots are small.
\end{theorem}

$2977$ knots with $12$ crossings or less are $1019$ large, and $1958$ small (\cite{BCT2012}).
For tunnel number $1$ knots, there exists a large knot having meridional peripherally incompressible essential closed surface in its exterior (\cite{EM2000}).
If any essential closed surface in the knot exterior $E(K)$ is meridional peripherally compressible, then $K$ is said to be {\em meridional}.

\begin{theorem}[\cite{LP1985},\cite{M1984},\cite{A1992},\cite{O1984},\cite{A1994},\cite{MO2010}]
$3$-braid knots, alternating knots, almost alternating knots, Montesinos knots, toroidally alternating knots, algebraically alternating knots are meridional.
\end{theorem}

Here, toroidally alternating knots are a class containing alternating knots and almost alternating knots, Montesinos knots, and algebraically alternating knots are a class containing algebraic knots and alternating knots.

We say that a surface $F$ embedded in the knot exterior $E(K)$ is {\em free} if each component of $E(K)-\text{int}N(F)$ is a handlebody.

\begin{theorem}[\cite{MOz2010}]\label{meridional_free}
Any non-meridional essential surface embedded in the exterior of a meridional knot is free.
\end{theorem}

Small knots are similarly defined also for 3-manifolds other than $S^3$.
A $3$-manifold which does not contain essential closed surface is said to be {\em small}.

\begin{conjecture}[The Lopez conjecture \cite{L1992}]
There exists a small knot in every small closed $3$-manifold.
\end{conjecture}

Some construction of small knots in small closed $3$-manifolds are given in \cite{M2004}, \cite{QW2006} etc.

\subsubsection{Invariants for essential closed surfaces}

\begin{definition}[\cite{MOz2010}]
We define the {\em waist} of a knot $K$ as following.
\[
waist(K)=\max_{F\in\mathcal{F}} \min_{D\in\mathcal{D}_F} |D\cap K|
\]
Here, $\mathcal{F}$ denotes the set of all incompressible closed surfaces in $S^3-K$, and $\mathcal{D}_F$ denotes the set of all compressing disks for $F$ in $S^3$.
We define $waist(K)=0$ for the trivial knot $K$.
A knot $K$ is meridional if and only if $waist(K)=1$.
\end{definition}

\begin{theorem}[\cite{MOz2010}]\label{waist}
It holds that $\displaystyle waist(K)\le \frac{trunk(K)}{3}$.
\end{theorem}

\begin{definition}[\cite{O2000}]
Let $S$ be a surface in a knot exterior $E(K)$.
Let $i:S\to E(K)$ be the inclusion map and $i_*:H_1(S)\to H_1(E(K))$ be the induced homomorphism.
Since $Im(i_*)$ is a subgroup of an infinite cyclic group $H_1(E(K))=\mathbb{Z}$ which has a meridian as a generator, there exists an integer $m$ such that $Im (i_*)= m \mathbb{Z}$.
Then, we define the {\em order} $o(S)$ of $S$ as $|m|$.
\end{definition}

For an essential closed surface $S$ in the exterior of a meridional knot, we have $o(S)=1$ as $waist(K)=1$.
And, for an essential closed surface in the fibered knot exterior, we have $o(S)\ne 0$ by Theorem \ref{non-free}.

\begin{theorem}[\cite{O2002}]\label{positive_closed}
For an essential closed surface $S$ in the positive knot exterior, we have $o(S)\ne 0$.
\end{theorem}

\subsubsection{Tunnel numbers}

For Heegaard splittings of knot exteriors, there is a survey in \cite{M2007}.
For torus knots (\cite{BRZ1988}), $2$-bridge knots (\cite{K1999}), alternating knots (\cite{L2005}), satellite knots (\cite{MS1991}, \cite{EM1994}), the classification of unknotting tunnels has been completed in the case that tunnel number $1$.
Satellite knots with tunnel number $1$ and genus $1$ were determined, and the following was conjectured (\cite{GT1999}).

\begin{theorem}[Goda--Teragaito conjecture \cite{S2004}]
A tunnel number $1$ and genus $1$ knot is either a satellite or $2$-bridge knot.
\end{theorem}

For a genus $2$ Heegaard surface which is obtained from an unknotting tunnel $\tau$, the Hempel distance $d(\tau)$ is defined.
If $d(\tau)>5$, then unknotting tunnels are unique (\cite{J2006}, \cite{ST2006}).

The set of tunnel number $1$ knots has a structure of a tree, and its bridge number is recursively determined.

\begin{theorem}[\cite{CM2009}]
Any tunnel number $1$ knot is obtained from the trivial knot by a unique sequence of cabling constructions.
\end{theorem}

\subsection{Tangle decomposing surfaces ($K\pitchfork F$)}

\subsubsection{Tangle decomposing spheres}

For alternating knots (\cite{M1984}), positive knots (\cite{O2002}), Montesinos knots (\cite{O1984}), essential $1$-string tangle decomposing spheres, namely decomposing spheres, are classified.
And essential $2$-string tangle decomposing spheres, namely essential Conway spheres, are classified for alternating knots (\cite{M1984}), Montesinos knots (\cite{O1984}), $3$-bridge knots (\cite{MO2009}).
An unknotting number $1$ knot has no decomposing sphere (\cite{S1985}).
For unknotting number $1$ knots, the relation with essential Conway spheres is classified (\cite{GL2006}).

\begin{theorem}[\cite{GR1995}, cf. \cite{M1997}, \cite{Teragaito}, \cite{OM1998}]\label{tunnel_tangle}
Tunnel number $1$ knots and free genus $1$ knots have no essential $n$-string tangle decomposing sphere for any $n$.
\end{theorem}

Theorem \ref{tunnel_tangle} does not hold for tangle decomposing spheres of genus greater than or equal to $1$ (\cite{EM2000}).
It is conjectured from Conjecture \ref{rank_genus} that $2$-generator knots are tunnel number $1$ knots.
Therefore, it is expected that Theorem \ref{tunnel_tangle} also holds for $2$-generator knots, but in fact, those knots have no decomposing sphere (\cite{N1982}), and have no essential Conway sphere (\cite{BW2005}).

\begin{theorem}[\cite{O1998k}]
Knots with essential free $2$-string tangle decompositions have no essential $n$-string tangle decomposition for any $n\ne2$.
Moreover, essential $2$-string tangle decompositions are unique.
\end{theorem}

\begin{theorem}[{\cite[Theorem 2.0.3]{CGLS1987}}]
Knots with essential tangle decomposing surfaces are large.
\end{theorem}

\begin{theorem}[\cite{T1997}]\label{thin_tangle}
If a knot in thin position is not in bridge position, then it has an essential tangle decomposing sphere.
\end{theorem}

In particular, most thin level sphere for a knot in a thin position is an essential tangle decomposing sphere (\cite{W2004}).
In general, if $ht(K)=n+1$, then there exist $n$ mutually disjoint non-parallel essential tangle decomposing spheres(\cite{B2003}, cf. \cite{HK1997}).
In thin position, a knot with an inessential thin level sphere is discovered (\cite{BZ2014}).

The stream of generalized Heegaard splittings for $3$-manifolds by Casson--Gordon (\cite{CG1987}), Scharlemann--Thompson (\cite{ST1994}) is extended to knots in $3$-manifolds by \cite{HS2001}, \cite{T2009}.

\begin{theorem}[\cite{HS2001}, \cite{T2009}]
If a knot has a bridge position of Hempel distance $1$, then it admits a Type I move, or the knot exterior contains a meridional essential surface.
\end{theorem}

\subsubsection{Bridge decomposing spheres}

A bridge position which admits no Type I move is said to be {\em locally minimal}.
For the trivial knot (\cite{Ot1982}), \cite{HS1998}, \cite{O2009}), $2$-bridge knots (\cite{Ot1985}, \cite{ST2008}), torus knots (\cite{OZ2011}, \cite{Z2015}), a locally minimal bridge position is globally minimal.
A knot with a locally minimal bridge position which is not globally minimal was firstly given in \cite{OT2013}.
Moreover, an example of knots with arbitrary gap between bridge numbers of locally minimal bridge positions is constructed (\cite{JKOT}).
On the other hand, for the trivial knot, there exists a locally minimal Morse position which is not thin position (\cite{Z2011}).
For hyperbolic knots, it is announced that the number of bridge decomposing spheres is finite for each bridge number (\cite{C2007}).
And the bridge number of Montesinos knots is determined (\cite{BZ1985}), arborescent knots of bridge number $3$ are characterized (\cite{J2011}).

\begin{theorem}[\cite{BS2005}, cf. \cite{J2014}] \label{distance}
Let $K$ be a knot in bridge position with respect to a bridge decomposing sphere $F$, and $S$ be an essential surface properly embedded in $E(K)$.
Then $d(K,F)\le 2g(S) +|\partial S|$ holds.
\end{theorem}

If the Hempel distance is large with respected to a globally minimal bridge decomposing sphere, then the minimal decomposition is unique (\cite{T2007}).
And if the Hempel distance is large, then it does not admit exceptional Dehn surgeries (\cite{BCJTT201209}).

By Theorem \ref{tunnel_tangle} and \ref{thin_tangle}, if a tunnel number $1$ knot is in thin position, then it is in globally minimal bridge position.

\begin{theorem}[Morimoto conjecture \cite{GST2000}]\label{leveling_tunnel}
If a tunnel number $1$ knot is in globally minimal bridge position, then the unknotting tunnel can be moved by slides and isotopies so that it is contained in the bridge decomposing sphere.
\end{theorem}

By Theorem \ref{leveling_tunnel} and \cite{M1992}, the classification theorem (\cite{K1999}) for unknotting tunnels of $2$-bridge knots follows.

\subsection{Seifert surfaces ($K=\partial F$)}

\subsubsection{Kakimizu complexes}

By giving a handle decomposition between two mutually interior-disjoint Seifert surfaces similarly to \ref{Morse--Novikov}, it turns out from Theorem \ref{disjoint_minimal_Seifert} together that any Seifert surfaces are stably equivalent.

\begin{theorem}[{\cite{ST1988}}, \cite{K1992}]\label{disjoint_minimal_Seifert}
For any two (minimal genus) Seifert surfaces $F$ and $F'$ for a knot, there exists a sequence of (minimal genus) Seifert surfaces $F=F_0,F_1,\ldots,F_n=F'$ such that $intF_i\cap intF_{i-1}=\emptyset$ for each $i$ $(1\le i\le n)$.
\end{theorem}

By Theorem \ref{disjoint_minimal_Seifert}, it follows that Kakimizu complexes (\cite{K1992}) are connected.
Moreover, Kakimizu complexes are simply connected (\cite{S2010}) and contractible (\cite{PS2012}).
For a hyperbolic knot of genus $1$, the number of mutually disjoint non-parallel genus $1$ Seifert surfaces is at most $7$ (\cite{T2003}).
On the other hand, for any natural number $n$, there exists a hyperbolic knot of genus $1$ which bounds mutually disjoint non-parallel Seifert surface of genus $1$ and $n$ Seifert surfaces of genus $2$ (\cite{T2003}).
And for a hyperbolic knot $K$, the diameter of the Kakimizu complex is at most $2g(K)(3g(K)-2)+1$ (\cite{SS2009}).

\subsubsection{Seifert surfaces and closed surfaces}

\begin{theorem}[\cite{O2000}]\label{non-free}
A knot bounds a non-free incompressible Seifert surface if and only if there exists an essential closed surface of order $0$ in the knot exterior.
\end{theorem}

Hence, by Theorem \ref{meridional_free}, \ref{positive_closed}, incompressible Seifert surfaces for meridional knots and positive knots are free.
The next theorem can be obtained by applying Haken's normal surface theory (\cite{H1961}) to Seifert surfaces for knots.

\begin{theorem}[\cite{W2008}]
For a knot $K$, there exist a finite number of incompressible Seifert surfaces $\{ S_1,\ldots,S_n\}$ and a finite number of essential closed surfaces $\{ Q_1,\ldots,Q_m\}$ such that any incompressible Seifert surface $S$ is represented by a Haken sum $S=S_i+a_1Q_1+\cdots+a_mQ_m$ $(a_1,\ldots,a_m\ge 0)$.
\end{theorem}

Hence, the number of incompressible Seifert surfaces for a small knot is finite.
Moreover, if we restrict Seifert surfaces to minimal genus, then by Theorem \ref{finiteness} or \cite{SS1964}, \cite{W2008} it follows the finiteness for hyperbolic knots.

\subsubsection{Murasugi sums}

\begin{definition}[\cite{M1963}, cf. \cite{S1978}]
Let $F$ be a surface whose boundary is a link $K$.
Suppose that there exists a sphere $S$ separating $S^3$ into two $3$-balls $B_1,\ B_2$ such that $F\cap S$ is a single disk.
Put $F_i=F\cap B_i$ for $i=1,\ 2$.
Then, $F$ is said to be {\em Murasugi decomposed} into $F_1$ and $F_2$, denoted by $F=F_1*F_2$.
Conversely, we say that $F$ is obtained from $F_1$ and $F_2$ by a {\em Murasugi sum} along a disk $F\cap S$.
In particular, if $F\cap S$ is a quadrangle, then the Murasugi sum is called a {\em plumbing}, the Murasugi decomposition is called a {\em deplumbing}.
\end{definition}

\begin{theorem}[\cite{G1983}, \cite{G1985}, \cite{S1978}]\label{Gabai_Murasugi}
Let $F=F_1*F_2$ be a Seifert surface.
If $F_1$ and $F_2$ is incompressible (minimal genus, fibered), then $F$ is also  incompressible (minimal genus, fibered).
The converse holds for minimal genus, fibered.
\end{theorem}

The incompressibility in Theorem \ref{Gabai_Murasugi} is extended to non-orientable surfaces (\cite{O2011}).

Harer gave a method to construct all fibered links, and conjectured next (\cite{H1982}).

\begin{theorem}[Harer conjecture \cite{GG2006}]
Any fibered link is obtained from the trivial knot by a sequence of plumbing and deplumbing of Hopf bands.
\end{theorem}

For general links, the next sequence is given.

\begin{theorem}[\cite{T1989}]
Let $L$ be a non-splittable link.
Then there exists a sequence of triple $(L_0,S_0,\alpha_0)\to (L_1,S_1,\alpha_1)\to \cdots \to (L_{m-1},S_{m-1},\alpha_{m-1})\to L_m=L$ which satisfies the followings.
\begin{enumerate}
\item $L_0$ is the trivial knot.
\item $S_i$ is a Seifert surface of maximal Euler characteristic for $L_i$ ($i=0,\ldots,m-1$).
\item $\alpha_i$ is an arc properly embedded in $S_i$.
\item $S_{i+1}$ is obtained from $S_i$ by a twisting along $\alpha_i$ or a plumbing of a Hopf band along $\alpha_i$.
\end{enumerate}
\end{theorem}

\subsubsection{Seifert surfaces and crossing changes}

Suppose that three oriented links $L_+,\ L_-,\ L_0$ are in skein relation, and let $\chi(L)$ be the maximal Euler characteristic among all Seifert surfaces for $L$.

\begin{theorem}[\cite{ST1989}]
If we arrange $\chi(L_+),\ \chi(L_-),\ \chi(L_0)-1$ in ascending order, then we have $\chi_1=\chi_2\le\chi_3$.
In particular, if $\chi(L_+)=\chi(L_0)-1<\chi(L_-)$, then there exist maximal Euler characteristic Seifert surfaces $S'$ and $S$ for $L_+$ and $L_0$ such that $S'$ is obtained from $S$ by a plumbing of a Hopf band.
\end{theorem}

For a knot $K$, take an arc $\alpha$ satisfying $\alpha\cap K=\partial \alpha$.
We say that a crossing change of $K$ along $\alpha$ is {\em nugatory} if there exists a sphere intersecting $K$ transversely such that $S\cap (K\cup \alpha)=\alpha$.

\begin{conjecture}[Nugatory crossing conjecture {\cite[Problem 1.58]{K1995}}]\label{nugatory}
If $K$ is obtained from $K$ by a crossing change, then the crossing is nugatory.
\end{conjecture}

Conjecture \ref{nugatory} holds for the trivial knot (\cite{ST1989}), $2$-bridge knots (\cite{T1999}), fibered knots (\cite{K2012}), and other knot classes (\cite{BFKP2012}, \cite{BK2013}).

\subsection{Coiled surfaces ($K\subset F$)}

\subsubsection{Neuwirth conjecture}

\begin{conjecture}[Neuwirth conjecture, \cite{N1964}]
For any non-trivial knot $K$, there exists a closed surface $F$ such that it contains $K$ non-separatingly and it is essential in the exterior of $K$.
\end{conjecture}

\begin{conjecture}[Strong Neuwirth conjecture, \cite{OR2012}, cf. \cite{IOT2002}]
For any prime knot except for torus knots, there exists an essential non-orientable spanning surface\footnote{After writing the original article in Japanese, Dunfield found a counterexample for the Strong Neuwirth conjecture (\cite{ND2015}).}.
\end{conjecture}

The Neuwirth conjecture holds for alternating knots (\cite{A1956}) and generalized alternating knots (\cite{O2006}), knots satisfying $g_I(K)<2g(K)$ for the interpolating genus $g_I(K)$ (\cite{N1964}), knots with non-orientable spanning surfaces obtained by Murasugi sums of essential spanning surfaces (\cite{O2011}), Montesinos knots (\cite{OR2012}), knots with $11$ crossings or less except for $K11_n118$ and $K11_n126$ (\cite{OR2012}), knots with degree $1$ mappings to knots satisfying the Neuwirth conjecture (\cite{OR2012}), almost all generalized arborescently alternating knots (\cite{OR2012}), uniformly twisted knots (\cite{O2012}).

\subsubsection{Invariants for coiled surfaces}

\begin{definition}[\cite{O2009}]
Let $F$ be a closed surface containing a knot $K$.
We define the {\em representativity} of a pair $(F,K)$ as $\displaystyle r(F,K)=\min_{D\in\mathcal{D}_F} |\partial D\cap K|$, where $\mathcal{D}_F$ denotes the set of all compressing disks for $F$.
If $F$ is a sphere and $K$ is the trivial knot, we define as $r(F,K)=1$.
Moreover, we define the {\em representativity} of a knot $K$ as 
\[
r(K)=\max_{F\in\mathcal{F}} r(F,K),
\]
where $\mathcal{F}$ denotes the set of all closed surfaces containing $K$.
\end{definition}

It holds that $r(K)=1$ for the trivial knot $K$, and $r(K)\ge 2$ for a non-trivial knot $K$.
Moreover, it holds that $r(K)=2,\ =\min\{p,q\},\ \le 3$ for $2$-bridge knots, $(p,q)$-torus knots, algebraic knots $K$ respectively, and the pretzel knots of representativity $3$ are determined (\cite{O2012p}).

\begin{conjecture}[\cite{O2009}]
For an alternating knot $K$, it holds that $r(K)=2$\footnote{Kindred solved this conjecture affirmatively in \cite{K2017}.}.
\end{conjecture}

\begin{theorem}[\cite{O2009}, \cite{BO2017}]\label{representativity_bridge}
It holds that $r(K)\le trunk(K)/2$.
\end{theorem}

A knot has no essential $n$-string tangle decomposing sphere for $n<r(K)/2$ (\cite{O2009}).

\begin{definition}[\cite{O2009}]
We define the {\em closed genus} for a non-trivial knot $K$ as
\[
cg(K)=\min_{F\in \mathcal{F}} \{g(F)|r(F,K)\ge 2\},
\]
where $\mathcal{F}$ denotes the set of all closed surfaces containing $K$.
For the trivial knot $K$, we define as $cg(K)=0$.
\end{definition}

In a similar way as the additivity of the genera for knots, it turns out that the closed genera are additive with respect to connected sums for knots.

\begin{definition}[\cite{M1994}]
For any knot $K$, there exists a Heegaard surface $F$ of $S^3$ containing $K$.
We call the minimal genus of these Heegaard surfaces $F$ the {\em $h$-genus} of $K$, denoted by $h(K)$.
\end{definition}

\begin{theorem}[\cite{M1994}]
It holds that $t(K)\le h(K) \le t(K)+1$.
\end{theorem}

Corresponding to the stably equivalent theorem for Heegaard splittings by Reidemeister--Singer (\cite{R1933}, \cite{S1933}), also for Heegaard surfaces containing a knot $K$, if the surface slopes coincide, then they are $K$-stably equivalent (\cite{AS2009}).

\section{Conclusion}

There is a close relationship between incompressible surfaces, Morse functions and generalized Heegaard splittings.
One can see the existence of Heegaard splittings by the neighborhood of $n$-simplexes $(n=0,1,2,3)$ as $n$-handles in a triangulation of a $3$-manifold.
Here, both of the union of all $0$, $1$-handles and the union of all $2$, $3$-handles are handlebodies, and the surface separating them is a Heegaard surface.
If a $1$-handle is disjoint from a $2$-handle (weakly reducible Heegaard splitting \cite{CG1987}), then we can exchange the order of them.
The generalized Heegaard splitting which is obtained by exchanging $1$-handles and $2$-handles as much as possible is called a {\em locally thin position}.
In a locally thin position, each Heegaard surface is strongly irreducible, and locally thin level surfaces are incompressible (\cite{ST1994}).
An $i$-handle in a generalized Heegaard splitting corresponds to the index $i$ critical point for a Morse function on a $3$-manifold (\cite{M2002}).
Therefore, we can obtain an incompressible surface as a locally thin level surface for a Morse function (locally thin Morse position) corresponding to a locally thin position.
Conversely, for any incompressible surface, we can construct a locally thin Morse position such that it is a locally thin level surface.
From the above, considering incompressible surfaces arrives at considering Morse functions corresponding to locally thin positions.

At the present, although incompressible surfaces and strongly irreducible Heegaard surfaces are the main stream for surfaces in $3$-manifolds, a concept which extends them is proposed by Bachman (\cite{B2012a}).
For a separating surface in a $3$-manifold, by regarding isotopy classes of all compressing disks as vertices, and if compressing disks corresponding to $m+1$ vertices can be taken so that they are mutually disjoint, then they bound an $m$-simplex, and a disk complex is obtained.
The minimal number $n$ such that $n-1$-dimensional homotopy group of the disk complex is non-trivial is called the homotopy index.
If the disk complex is an empty set, then we define the homotopy index as $0$.
For a surface whose disk complex is an empty set or non-contractible, we define the topological index as the homotopy index of the disk complex.
In this definition, a surface is incompressible, strongly irreducible, cirtical if and only if the topological index is $0$, $1$, $2$ respectively.
In general, for a surface of topological index $n$, a theorem which extends the essential transversality theorem is proved by Bachman, and it might be a main stream of studies on surfaces in $3$-manifolds after this.

In recent years ($10\sim 20$ years), in low dimensional topology, the direction which studies the structure of complexes naturally obtained from objects is a main stream.
There are two construction methods for complexes.
One is a method that for all objects which are mutually moved by finite some operations, by regarding each object as a $0$-simplex, and if all of $n+1$ objects are mutually moved by a single operation, then they bound an $n$-simplex.
For examples of complexes of this type, there are Gordian complex (\cite{HU2002}), Reidemeister complex (\cite{M2012}), bridge complex, width complex (\cite{S2009}) etc.
Another is a method that by regarding each object as a $0$-simplex, and if all $n+1$ objects are mutually disjoint, then they bound an $n$-simplex.
For examples of complexes of this type, there are curve complex (\cite{H1981}), Kakimizu complex (\cite{K1992}) etc.

Finally, we stated a survey on surfaces with respect to knots in this article, but we regret that we cannot cover all definitions, results, references.
As texts which are written on surfaces embedded in $3$-manifolds, we cite \cite{M1996}, \cite{HM1995} for Japanese literatures, \cite{JS2014}, \cite{FM1997}, \cite{L1999}, \cite{L2014}, \cite{JJ0},  \cite{AH1}, \cite{SF1} for English literatures.
For more developed methods, we cite \cite{S1989}, \cite{G}, \cite{S2002}, \cite{PBS}, \cite{AFW}, \cite{M2016}.

\bigskip

\noindent{\bf Acknowledgements.}

The author would like thank to Yeonhee Jang for her comments on Conjectures \ref{rank_genus} and \ref{meridian_bridge}.

\bibliographystyle{amsplain}

\end{document}